%% file: paper.tex
\documentclass[leqno,11pt,a4paper]{amsart}

\usepackage{amsmath}
\usepackage{amsthm}
\usepackage{cite}
\usepackage{enumitem}
\usepackage{hyperref}
\usepackage{graphicx}
\usepackage{amssymb}
\usepackage{xcolor}
\hypersetup{
  colorlinks,
  linkcolor={blue!50!black},
  citecolor={blue!50!black},
  urlcolor={blue!80!black}
}
\usepackage{multirow}
\usepackage[all]{xy}
\usepackage{tikz}
\usepackage{booktabs}
\usepackage{arydshln}
\usepackage{mleftright}
\usepackage{mathrsfs}

\setlist[enumerate]{labelsep=*, leftmargin=1.5pc}
\setlist[enumerate]{label=\normalfont(\roman*), ref=\roman*}
\usepackage[margin=2.7cm]{geometry}
\graphicspath{{images/}}
\theoremstyle{plain}
\newtheorem{thm}{Theorem}[section]
\newtheorem{pro}[thm]{Proposition}
\newtheorem{lem}[thm]{Lemma}
\newtheorem{cor}[thm]{Corollary}

\theoremstyle{definition}
\newtheorem{dfn}[thm]{Definition}
\newtheorem{rem}[thm]{Remark}
\newtheorem{eg}[thm]{Example}
\newtheorem{notation}[thm]{Notation}

\DeclareMathOperator{\Hom}{Hom}

\DeclareMathOperator{\Pic}{Pic}

\DeclareMathOperator{\GL}{{GL}}

\newcommand{\cB}{\mathcal{B}}

\newcommand{\cG}{\mathcal{G}}

\newcommand{\cM}{\mathcal{M}}

\newcommand{\CC}{\mathbb{C}}

\newcommand{\TT}{\mathbb{T}}

\newcommand{\ZZ}{{\mathbb{Z}}}

\renewcommand{\tilde}{\widetilde}

\begin{document}
\author[T.\,Prince]{Thomas Prince}
\address{Mathematical Institute\\University of Oxford\\Woodstock Road\\Oxford\\OX2 6GG\\UK}
\email{thomas.prince@magd.ox.ac.uk}

\title[Mutation invariance]{Mutation equivalence of toric Landau--Ginzburg models}
\maketitle
\begin{abstract}
	Given a Fano complete intersection defined by sections of a collection nef line bundles $L_1,\ldots, L_c$ on a Fano toric manifold $Y$, a construction of Givental/Hori--Vafa provides a mirror-dual Landau--Ginzburg model. This construction depends on a choice of suitable nef partition; that is, a partition of the rays of the fan determined by $Y$. We show that toric Landau--Ginzburg models constructed from different nef partitions representing the same complete intersection are related by a volume preserving birational map. In particular, various Laurent polynomial mirrors which may be obtained from these Landau--Ginzburg models are mutation equivalent.
\end{abstract}

\input{introduction}
\input{hypersurfaces}
\input{complete_intersections}
\input{amenable}

\bibliographystyle{plain}
\bibliography{bibliography}
\end{document}

%% file: introduction.tex
\section{Introduction}
\label{sec:introduction}

We fix an $n$-dimensional toric Fano manifold $Y$ and nef line bundles $L_1,\ldots, L_c \in \Pic(Y)$ such that the line bundle
\[
\Lambda := -K_Y - L_1 - \cdots - L_c
\]
is ample. We recall some standard results on the toric geometry of $Y$ to fix notation. In particular, we recall that the toric variety $Y$ contains an open dense torus $\TT_N := N \otimes_\ZZ (\CC^\star)$, where $N$ is an $n$-dimensional lattice. Letting $M$ denote the dual lattice to $N$, we recall the following exact sequence associated to $Y$:
\[
0 \longrightarrow M \stackrel{\rho^\star}{\longrightarrow} (\ZZ^R)^\star \stackrel{D}{\longrightarrow} \Pic(Y) \longrightarrow 0.
\]
Here $R$ is the number of rays of the fan $\Sigma$ determined by $Y$ and $\rho^\star$ is dual to the map $\rho$ which sends elements of the standard basis of $\ZZ^R$ to the ray generators $\{v_1, \ldots, v_R\}$ of $\Sigma$. Setting $D_i := D(e_i)$, a \emph{nef partition} of $\Sigma$ is a partition of $\{v_1, \ldots, v_R\}$ into $c+1$ sets $S_i$ such that, for each $i \in \{1,\ldots, c\}$, we have that $L_i = \sum_{j \in S_i} D_j$.

Given such a nef partition, Givental~\cite{Givental:toric} and Hori--Vafa~\cite{Hori--Vafa} describe a mirror toric Landau--Ginzburg model; this is a pair $(\breve{X},W)$ such that
\[
\breve{X} := \Big\{\sum_{j \in S_1}{z^{v_j}} = 1, \ldots, \sum_{j \in S_c}{z^{v_j}} = 1\Big\} \subset M\otimes_\ZZ \CC^\star
\]
and,
\[
W := z^{v_1} + \cdots + z^{v_n},
\]
where $z^v$ denotes the monomial in $\CC[N]$ corresponding to an element $v$ of $N$. Letting $z_1, \ldots, z_n$ denote co-ordinates on $\TT_N$, we also fix the holomorphic volume form 
\[
\omega_n := \bigwedge_{i=1}^n d\log(z_i).
\]

\begin{thm}
	\label{thm:main_result}
	Given a pair of nef partitions $\{S_i : 1 \leq i \leq c+1\}$ and $\{S'_i : 1\leq i \leq c+1\}$ of $(Y;L_1,\ldots, L_c)$, let $(\breve{X},W)$ and $(\breve{X}',W)$ denote the corresponding pair of mirror toric Landau--Ginzburg models. There is a birational map
	\[
	\varphi \colon \TT_N \dashrightarrow \TT_N
	\]
	such that $\varphi^\star(W) = W$, $\varphi^\star\omega_n = \omega_n$, and
	\[
	\varphi^\star\Big(\sum_{j \in S_i}{z^{v_j}}\Big) = \sum_{j \in S'_i}{z^{v_j}}
	\]
	 for all $i \in \{1,\ldots, c\}$.
\end{thm}
Since $\omega_n$ is a holomorphic volume form, we say that birational maps $\varphi$ such that $\varphi^\star\omega_n = \omega_n$ are \emph{volume preserving}. We note that there is a close analogy between Theorem~\ref{thm:main_result} and results of Li \cite{L16} and Clarke \cite{C17}, which establish the birational invariance of Calabi--Yau mirror-duals given by different nef partitions in the (original) sense of Batyrev and Borisov \cite{BB96,B93}. It would be interesting to explore whether our methods can be directly related to those used in the Calabi--Yau setting, noting for example the mutual generalisation of mirror models to Fano and Calabi--Yau toric complete intersections given by Clarke~\cite{C17:Duality}.

As a corollary of Theorem~\ref{thm:main_result}, we can establish the mutation equivalence of various possible Laurent polynomial mirrors to toric Fano complete intersections. In particular, given an \emph{amenable collection subordinate to a nef partition}, as introduced by Doran--Harder~\cite[Definition~$2.2$]{Doran--Harder}, one can form a birational map $(\CC^\star)^{n-c} \dashrightarrow \breve{X}$, and hence define a Laurent polynomial $f$ on the complex torus $(\CC^\star)^{n-c}$ by pulling-back $W$. Throughout this article we say that a pair of Laurent polynomials $f$ and $g$ are \emph{mutation equivalent} if there is a volume preserving birational map $\varphi$ such that $\varphi^\star f =g$.

\begin{cor}[Cf.\ {\cite[Theorem~$2.24$]{Doran--Harder}}]
	\label{cor:amenable_collection}
	Laurent polynomial mirrors obtained from amenable collections subordinate to nef partitions of $(Y; L_1, \ldots, L_c)$ are mutation equivalent.
\end{cor}

In particular, the smoothness of the toric Fano variety $Y$ plays no role in the proof of Theorem~\ref{thm:main_result}. However, as a special case of Corollary~\ref{cor:amenable_collection}, we conclude \cite[Theorem~$5.1$]{CKP15}. We note that it is stated in \cite{CKP15} that Theorem~$5.1$ is proved in \cite{Prince:PhD}, however this erroneously states that the result follows directly from \cite[Theorem~$2.24$]{Doran--Harder} and we thank N. Ilten for bringing this to our attention. To state this result we first fix a nef partition of $(Y;L_1,\ldots, L_c)$, and choose distinguished elements $s_i \in S_i$ for each $i \in \{1,\ldots,c\}$ and a subset $E \subset S_{c+1}$ such that:
\begin{itemize}
	\item The set $\{D_i : i \in E\}$ forms a basis of $\Pic(Y)$.
	\item We have that $L_i$ is a non-negative linear combination of $\{D_j : j \in E\}$ for each $i \in \{1,\ldots, c\}$.
\end{itemize}

In \cite[\S$5$]{CKP15} an algorithm -- the \emph{Przyjalkowski method} -- for forming a Laurent polynomial $f$ from a nef partition $\{S_i : 1\leq i \leq c+1\}$ together with choices of $s_i \in S_i$ and $E$ is described. We let $X$ denote a complete intersection in $Y$ defined by sections of  $L_1, \ldots, L_c$.

\begin{cor}[Cf.\ {\cite[Theorem~$5.1$]{CKP15}}]
	\label{cor:CKP15}
	Let $f$ and $g$ be Laurent polynomial mirrors to $X$, formed by applying the Przyjalkowski method to a pair of nef partitions $\{S_i : 1 \leq i \leq c+1\}$, together with choices of elements $s_i \in S_i$ and subsets $E \subset S_{c+1}$ as above; then the Laurent polynomials $f$ and $g$ are mutation equivalent.
\end{cor}

We first describe the proof of Theorem~\ref{thm:main_result} in the hypersurface case $(c=1)$.

---------------------------------------

\subsection*{Acknowledgements}

I thank Tom Coates, Alexander Kasprzyk, and Nathan Ilten for our conversations about mutations and toric Landau--Ginzburg models. TP was supported by a Fellowship by Examination at Magdalen College, Oxford.

%% file: hypersurfaces.tex
\section{The hypersurface case}
\label{sec:hypersurfaces}

Fix a toric Fano manifold $Y$, a line bundle $L$, and a pair of distinct sets $S$ and $S' \subset \{1,\ldots, R\}$ -- using the notation introduced in \S\ref{sec:introduction} -- such that
\[
L = \sum_{i \in S}{D_i} = \sum_{j \in S'}{D_j}.
\]
Given a ray generator $v_i$ of $\Sigma$, we let $X_i$ denote the variable in the total homogeneous co-ordinate ring of $Y$ corresponding to $e_i \in \ZZ^R$; recalling that $v_i$ is equal to $\rho(e_i)$. Crucially, we have that both $\prod_{i \in S}{X_i}$ and $\prod_{j \in S'}{X_j}$ are elements of $\Gamma(Y,L)$. In other words, the monomial
\[
\frac{\prod_{i \in S}{X_i}}{\prod_{j \in S'}{X_j}}
\]
is a rational function on $Y$, regular on $\TT_N$. Since $M = \Hom(N,\ZZ)$ is the character lattice of $\TT_N$, there is a linear function $w \colon N \to \ZZ$ such that
\begin{equation}
\label{eq:weight_vector}
w(v_i) = 
\begin{cases}
1 & \text{if $i \in S\setminus S'$}\\
-1 & \text{if $i \in S' \setminus S$} \\
0 & \text{otherwise.}
\end{cases}
\end{equation}

\begin{rem}
	\label{rem:weight_vectors}
	We note that the construction of $w$ is not restricted to the hypersurface case: given $(Y;L_1,\ldots, L_c)$, and a pair of nef partitions $\{S_i : 1 \leq i \leq c+1\}$ and $\{S'_i : 1 \leq i\leq c+1\}$, there is a unique $w_i$ for each $i \in \{1,\ldots c\}$ such that $w_i(v_j) = 1$ if $j \in S_i\setminus S'_i$, $w_i(v_j) = -1$ if $j \in S'_i \setminus S_i$, and such that $w_i$ evaluates to zero on all other ray generators.
\end{rem}

The second geometric ingredient which will play a key role in the proof of Theorem~\ref{thm:main_result} are certain \emph{integral reflections}.

\begin{dfn}
	Given a linear function $w \colon N \to \ZZ$, fix a pair of vectors $u^{\pm} \in w^{-1}(\pm1)$ respectively. We let $T_{u^+,u^-}$ denote the unique $\GL_n(\ZZ)$ transformation which fixes $w^{-1}(0)$ and sends $u^\pm$ to $u^\mp$. We refer to such lattice isomorphisms as integral reflections.
\end{dfn}

The isomorphism $T_{u^+,u^-}\colon N \to N$ defines an isomorphism $\CC[N] \to \CC[N]$, which we also denote by $T_{u^+,u^-}$. We next recall the notion of \emph{(algebraic) mutation}, as described by Akhtar--Coates--Galkin--Kasprzyk~\cite{ACGK}. These maps were introduced by Galkin--Usnich~\cite{Galkin--Usnich} as \emph{mutations of potential}.

\begin{dfn}
	\label{dfn:mutation}
	Fix a lattice vector $w \in M \cong \ZZ^n$, and a Laurent polynomial $F \in \CC[w^\bot]$. We let $\mu_{w,F} \colon \TT_N \dashrightarrow \TT_N$ be the birational map defined by setting
	\[
	\mu^\star_{w,F}(z^v) = z^vF^{-\langle w,v\rangle}
	\]
	for each $v \in N$. Given a Laurent polynomial $f \in \CC[N]$ such that $\mu^\star_{w,F}f \in \CC[N]$ we say that $\mu_{w,F}$ defines a mutation of $f$.
\end{dfn}

We recall that, for any $w$ and $F$ as in Definition~\ref{dfn:mutation}, $\mu^\star_{w,F}(\omega_n) = \omega_n$; that is, algebraic mutations are volume preserving birational maps, or -- in the terminology used by Katzarkov--Przyjalkowski in~\cite{Katzarkov--Przyjalkowski} -- \emph{symplectomorphisms of cluster type}. Following \cite{ACGK}, we refer to $w$ as the \emph{weight vector} and $F$ the \emph{factor} of the mutation $\mu_{w,F}$.

\begin{pro}
	\label{pro:hypersurface_case}
	Theorem~\ref{thm:main_result} holds in the case $c=1$, setting $L_1 = L$.
\end{pro}
\begin{proof}
	We first let $w \in M$ be the vector given in \eqref{eq:weight_vector}, and choose elements $u^\pm$ in $w^{-1}(\pm 1)$. Consider the mutation $\mu_{w,F}$ of $W = z^{v_1} + \cdots + z^{v_R}$ where $F$ is given by
	\[
	F := \frac{1}{z^{u^+}}\Big(\sum_{i \in S\setminus S'}{z^{v_i}}\Big).
	\]
	We compute that
	\[
	W_1 := \mu^\star_{w,F}W = z^{u^+} + \frac{1}{z^{u^+}}\sum_{i \in S\setminus S'}{z^{v_i}}\sum_{j \in S'\setminus S}{z^{v_j}} + \sum_{i \notin S \ominus S'}{z^{v_i}}.
	\]
	Next consider the mutation $\mu_{-w,F'}$ of $W_1$ where
	\[
	F' := \frac{1}{z^{u^-}}\Big(\sum_{i \in S'\setminus S}{z^{v_i}}\Big).
	\]
	Setting $W_2 := \mu^\star_{-w,F'}W_1$, we compute that
	\[
	W_2 = \frac{z^{u^+}}{z^{u^-}}\sum_{j \in S'\setminus S}{z^{v_j}} + \frac{z^{u^-}}{z^{u^+}}\sum_{i \in S\setminus S'}{z^{v_i}} + \sum_{i \notin S \ominus S'}{z^{v_i}}.
	\]
	Moreover, we compute that
	\begin{align*}
		\mu_{-w,F'}^\star\circ \mu^\star_{w,F}\Big(\sum_{i \in S\setminus S'}{z^{v_i}}\Big) &= \mu_{-w,F'}^\star z^{u^+}\\
		&= \frac{z^{u^+}}{z^{u^-}}\Big(\sum_{i \in S'\setminus S}{z^{v_i}}\Big).
	\end{align*}
	Finally, since $w(v_i) = 1$ for any $i \in S \setminus S'$, we have that $T_{u^+,u^-}(v_i) = v_i - u^+ + u^-$; that is, $T_{u^+,u^-}(z^{v_i}) = \frac{z^{u^-}}{z^{u^+}}z^{v_i}$. Similarly, for any $i \in S' \setminus S$, $T_{u^+,u^-}(v_i) = v_i - u^- + u^+$. Thus the composition $\varphi := T_{u^+,u^-} \circ \mu^\star_{-w,F'} \circ \mu^\star_{w,F}$ satisfies the conditions required by the statement of Theorem~\ref{thm:main_result}. 
\end{proof}

In order to exploit the symmetry of the situation, the proof of the general case of Theorem~\ref{thm:main_result} is a generalisation of a rephrased version of the above proof. Indeed, let $G := \sum_{i \in S\setminus S'}{z^{v_i}}$ and let $G' = \sum_{i \in S'\setminus S}{z^{v_i}}$. Applying $\mu^\star_{w,F}$ to $W$ and $G$ we obtain $W_1$ and $z^{u^+}$ respectively. However, applying $\mu^\star_{-w,F'}$ to $W$ and $G'$ we obtain
\[
\mu^\star_{-w,F'}W = z^{u^-} + \frac{1}{z^{u^-}}\sum_{i \in S\setminus S'}{z^{v_i}}\sum_{j \in S'\setminus S}{z^{v_j}} + \sum_{i \notin S \ominus S'}{z^{v_i}}
\] and $\mu^\star_{-w,F'}G' = z^{u^-}$ respectively. We note that $T_{u^+,u^-}$ exchanges $z^{u^+}$ and $z^{u^-}$, while every term in either 
\[
\Big(\sum_{i \in S\setminus S'}{z^{v_i}}\sum_{j \in S'\setminus S}{z^{v_j}}\Big) \text{, or } \sum_{i \notin S \ominus S'}{z^{v_i}}
\]
is fixed by $T_{u^+,u^-}$. Hence we can describe $\varphi$ via the following diagram:
\begin{equation}
\label{eq:hypersurface_diagram}
\xymatrix{
 & \TT_N \ar^{T_{u^+,u^-}}[rr] \ar@{-->}_{\mu_{w,F}}[dl] & & \TT_N \ar@{-->}^{\mu_{-w,F'}}[dr] & \\
\TT_N  & & & & \TT_N \ar@{-->}_{\varphi}[llll]
}
\end{equation}

\begin{rem}
	\label{rem:degenerate_cases}
	We observe that the proof of Proposition~\ref{pro:hypersurface_case} breaks down if $S=S'$, and this case -- despite its triviality -- has to be considered separately. In particular, the vector $w$ becomes equal to zero and there are no possible choices of vectors $u^\pm$. The degeneration of the argument at this case generalises in higher codimension, see Definition~\ref{dfn:communicating}.
\end{rem}

%% file: complete_intersections.tex
\section{Complete intersections}
\label{sec:complete_intersections}

Turning to the general case, we fix nef partitions $\{S_1,\ldots, S_{c+1}\}$ and $\{S'_1,\ldots, S'_{c+1}\}$ of $(Y; L_1,\ldots, L_c)$. Following Remark~\ref{rem:weight_vectors}, we can fix (weight) vectors $w_1,\ldots, w_c$ such that
\begin{equation}
\label{eq:weight_vector_2}
w_i(v_j) = 
\begin{cases}
1 & \text{if $j \in S_i\setminus S'_i$}\\
-1 & \text{if $j \in S'_i \setminus S_i$} \\
0 & \text{otherwise,}
\end{cases}
\end{equation}
for all $i \in \{1,\ldots, c\}$. We also fix integral reflections, determined by elements $u^\pm_i$, analogous to the vectors $u^\pm$ which appear in the proof of Proposition~\ref{pro:hypersurface_case}. Unfortunately, as indicated in Remark~\ref{rem:degenerate_cases}, the existence of such vectors is not guaranteed. To describe how the special case described in Remark~\ref{rem:degenerate_cases} generalises, we introduce the following graph.

\begin{dfn}
	\label{dfn:communicating}
	Given a pair of nef partitions $\{S_1,\ldots, S_{c+1}\}$ and $\{S'_1,\ldots, S'_{c+1}\}$ of $(Y; L_1,\ldots, L_c)$ we let $\cG$ denote the graph with vertex set $\{1,\ldots, c+1 \}$. We include an edge between vertices $i$ and $j$ if and only if either $S_i \cap S'_j \neq \varnothing$ or $S'_i \cap S_j \neq \varnothing$. We refer to the connected components of $\cG$ as \emph{communicating sets}.
\end{dfn}

Fixing a communicating set $C = \{i_1, \ldots, i_{l+1}\}$, for some $l \in \{0,\ldots, c\}$, we fix integral reflections associated to elements in each communicating set.

\begin{lem}
	\label{lem:reflections}
	For each $j \in C \setminus \{i_{l+1}\}$ there is a pair of vectors $u^\pm_j \in N$ such that
	\[
	w_a(u^\pm_j) = \begin{cases}
	\pm 1 & \text{ if $a=j$} \\
	0 & \text{if $a \notin \{j,i_{l+1}\}$,}
	\end{cases}
	\]
	for all $a \in \{1,\ldots c\}$.
\end{lem}
\begin{proof}
	We proceed by induction on the length of the shortest path in $\cG$ from $j$ to $i_{l+1}$. If there is a path of length one then there is an element $v \in S_j \cap S'_{i_{l+1}}$ or $S'_j \cap S_{i_{l+1}}$. We take $u^+_j$ to be equal to $v$ or $-v$ in the first and second of these cases respectively and make a similar assignment for $u^-_j$. Next we fix a path from $j$ to $i_{l+1}$ and let $p$ denote the second vertex in this path (after $j$). By the definition of $\cG$, there is an element $v \in S_j \cap S'_p$ or $S'_j \cap S_p$ and by the induction hypothesis there are elements $u^\pm_p$ satisfying the conditions of the lemma. Thus we may take $u^+_j$ to be either $v + u^+_p$ or $u^+_p - v$.
\end{proof}

\begin{eg}
	We note that in the hypersurface case ($c=1$) we have two possibilities; either $\cG$ is connected and Lemma~\ref{lem:reflections} ensures the existence of a pair of vectors $u^\pm_1$, or $\cG$ consists of a disjoint pair of points. In the latter case the nef partitions coincide and Lemma~\ref{lem:reflections} becomes trivial.
\end{eg}

Writing
\[
W = \sum_{i \in S_1}{z^{v_i}} + \cdots + \sum_{i \in S_{c+1}}{z^{v_i}},
\]
we partition these terms into blocks corresponding to communicating sets. To simplify notation in what follows we fix a communicating set $C = \{i_1,\ldots, i_{l+1}\}$, define
\[
W_C = \sum_{j \in C}\sum_{i \in S_j}{z^{v_i}}.
\]
We then note that Theorem~\ref{thm:main_result} follows immediately from the following restricted version.
\begin{thm}
	\label{thm:alt_main_result}
	For each communicating set $C$, there is a birational map $\varphi_C \colon \TT_N \dashrightarrow \TT_N$ such that $\varphi_C^\star(W_C) = W_C$, $\varphi_C^\star\omega_n = \omega_n$,
	\[
	\varphi_C^\star\Big(\sum_{j \in S_i}{z^{v_j}}\Big) = \sum_{j \in S'_i}{z^{v_j}}
	\]
	for all $i \in C \setminus \{i_{l+1}\}$, and $\varphi_C^\star(z^{v_j}) = z^{v_j}$ for all $j \notin \bigcup_{i \in C}S_j$.
\end{thm}

\begin{proof}[Proof of Theorem~\ref{thm:main_result}]
To prove Theorem~\ref{thm:main_result} from this result we note that the composition of the birational maps $\varphi_C$ satisfies all the conditions of Theorem~\ref{thm:main_result}. Note in particular that if $\varphi_C^\star(W_C) = W_C$ and
\[
\varphi_C^\star\Big(\sum_{j \in S_i}{z^{v_j}}\Big) = \sum_{j \in S'_i}{z^{v_j}}
\] 
for all $i \in C \setminus \{i_{l+1}\}$ it automatically follows that 
\[
\varphi_C^\star\Big(\sum_{j \in S_{i_{l+1}}}{z^{v_j}}\Big) = \sum_{j \in S'_{i_{l+1}}}{z^{v_j}}.
\]
\end{proof}

For the remainder of this section we fix a communicating set $C$ and assume for notational convenience that 
\[
C = \{1,\ldots, l+1\}.
\]
To construct the birational map $\varphi_C$ we first construct a sequence of mutations $\{\mu_{w_i,F_i}\}_{i \in \{1,\ldots, l\}}$, analogous to these defined in the proof of Proposition~\ref{pro:hypersurface_case}. To do so we must fix factor polynomials $F_i \in \CC[w_i^\bot]$. We inductively define these Laurent polynomials as follows: for any $i \in C$, we set $G_i := \sum_{j \in S_i}{z^{v_i}}$. We define $G^{(0)}_i := G_i$ and set
\[
F_i := \frac{1}{z^{u^+_i}}\tilde{G}^{(i-1)}_i,
\]
for any $i \in \{1,\ldots, l\}$, where $\tilde{G}^{(j)}_i$ is equal to the sum of the terms $z^u$ in $G^{(j)}_i$ such that $\langle w_i,u \rangle = 1$, and $G^{(k)}_i := \mu^\star_{w_k,F_k} G^{(k-1)}_i$ for any $k \in \{1,\ldots, l\}$. Similarly, we write $W_C^{(0)} := W_C$ and define $W_C^{(j)} := \mu^\star_{w_j,F_j}W_C^{(j-1)}$ for all $j \in \{1,\ldots, l\}$. We note that 
\[
W_C^{(j)} = \sum_{i \in C}{G^{(j)}_i}
\]
for all $i \in C$ and $j \in \{1,\ldots, l\}$.

\begin{eg}
	Suppose that $l \geq 2$; we see that, since the exponent $u \in N$ of a term of $G_1$ satisfies $\langle w_1,u \rangle = 1$ if and only if $u=v_i$ for some $i \notin S'_1$, 
	\[
	F_1 = \frac{1}{z^{u^+_1}}\sum_{i \in S_1 \setminus S'_1}{z^{v_i}}.
	\]
	Thus we have that
	\[
	G^{(1)}_2 = \sum_{i \in S_2 \setminus S'_1}{z^{v_i}} + \frac{1}{z^{u^+_1}}\sum_{i \in S_1 \setminus S'_1}{z^{v_i}}\sum_{j \in S_2 \cap S'_1}{z^{v_j}},
	\]
	and hence 
	\[
	F_2 = \frac{1}{z^{u^+_2}}\tilde{G}^{(1)}_2 =  \frac{1}{z^{u^+_2}}\sum_{i \in S_2 \setminus (S'_1 \cup S'_2)}{z^{v_i}} + \frac{1}{z^{u^+_1}z^{u^+_2}}\sum_{i \in S_2 \cap S'_1}{z^{v_i}}\sum_{j \in S_1 \setminus (S'_1 \cup S'_2)}{z^{v_j}},
	\]
\end{eg}
Evidently these polynomials rapidly become too complicated to record and manipulate directly in standard notation. To describe the factor polynomials $F_i$, and the effect of the corresponding mutations on $W$, we introduce new notation for these polynomials.

\begin{notation}
	\label{not:condensed_polynomials}
	For any $i,j \in C$, we define the polynomial $[i,j]$ as follows
	\[
		[i,j] := \frac{1}{z^{u^+_i}}\sum_{k \in S_i \cap S'_j}{z^{v_k}}.
	\]
	Moreover, given elements $a_1,\ldots, a_p \in C$ we introduce the notation
	\begin{align*}
	(a_1,\ldots, a_p) &:= z^{u^+_{a_1}}[a_1,a_2][a_2,a_3]\cdots [a_{p-1},a_p] \\ &= \frac{1}{\prod^{p-1}_{i=2}{z^{u^+_{a_i}}}}\Big(\sum_{k \in S_{a_1} \cap S'_{a_2}}{z^{v_k}}\Big)\Big(\sum_{k \in S_{a_2} \cap S'_{a_3}}{z^{v_k}}\Big)\cdots \Big(\sum_{k \in S_{a_{p-1}} \cap S'_{a_p}}{z^{v_k}}\Big).
	\end{align*}
\end{notation}

In this notation, $G^{(1)}_2$ has the form
	\begin{equation}
	\label{eq:G12}
	G^{(1)}_2 = \sum_{i > 1}{(2,i)} + \sum_{i > 1}{(2,1,i)}
\end{equation}
and $F_2$ has the form
\begin{align}
\label{eq:F2}
\nonumber \sum_{i \notin \{1,2\}}{[2,i]} + [2,1]\sum_{i \notin \{1,2\}}{[1,i]}\\ = \frac{1}{z^{u^+_2}}\Big(\sum_{i \notin \{1,2\}}{(2,i)} + \sum_{i \notin \{1,2\}}{(2,1,i)}\Big).
\end{align}

\begin{lem}
	\label{lem:evaluation}
	Let $u$ be an exponent of a term of the polynomial corresponding to a sequence $(a_1,\ldots, a_p)$ and suppose $i > \max(a_2,\ldots, a_{p-1})$. Then we have that
	\[
	\langle u,w_i\rangle = \begin{cases}
	1 & a_1 = i, a_p \neq i\\
	-1 & a_1 \neq i, a_p=i \\
	0 & \text{otherwise}
	\end{cases}
	\] 
\end{lem}
\begin{proof}
	Considering the expression
	\[
		(a_1,\ldots, a_p) = \frac{1}{\prod^{p-1}_{j=2}{z^{u^+_{a_j}}}}\Big(\sum_{k \in S_{a_1} \cap S'_{a_2}}{z^{v_k}}\Big)\Big(\sum_{k \in S_{a_2} \cap S'_{a_3}}{z^{v_k}}\Big)\cdots \Big(\sum_{k \in S_{a_{p-1}} \cap S'_{a_p}}{z^{v_k}}\Big),
	\]
	we note that $w_i$ evaluates to zero on all the vectors $u^+_j$. Moreover, $w_i$ evaluates to $\pm 1$ on vectors $v_j$ for $j \in S_i$ and $S'_i$ respectively and evaluates to zero on all other ray generators, from which the result follows.
\end{proof}

It may be easily shown by induction that, up to a product with a monomial, the polynomials $G^{(j)}_i$ and $F_j$ for $i \in \{1,\ldots, l+1\}$ and $j \in \{1,\ldots, l\}$ are sums of terms of the form $P$, where $P$ is a polynomial represented by a sequence, as described in Notation~\ref{not:condensed_polynomials}. We introduce an additional combinatorial notion in order to give a more detailed description of these polynomials.

\begin{dfn}
	We say that a sequence of numbers $a_1,\ldots, a_p$ in $\ZZ_{> 0}$ is \emph{hilly} if whenever $a_i = a_j$ for a pair $j > i$, there is a $k \in \{i, \ldots, j\}$ such that $a_k > a_i$. We note that the empty sequence is hilly.
\end{dfn}

For any $0 \leq j < i \leq l$, we let $\cM(i,j)$ denote the class of sequences $(a_1,\ldots, a_p)$ such that $a_1 = i$, $a_k \leq j$ for all $k \in \{2, \ldots, p-1\}$, $a_p > j$, and the sequence $(a_2,\ldots, a_{p-1})$ is hilly (noting that it may be empty).
 
\begin{lem}
	\label{lem:intermediate_description}
	For any $0 \leq j < i \leq (l+1)$, $G^{(j)}_i$ is a sum of terms of the form $P$, where
	\[
	P = (a_1,\ldots, a_p) \in \cM(i,j).
	\]
	The factors $F_i$ are sums of terms of the form $z^{-u^+_i}P$, where $P \in \cM(i,i-1)$ and $a_p > i$. Moreover, any $P$ of the class described appears in precisely one term of each of $G^{(j)}_i$ and $F_i$.
\end{lem}
\begin{proof}
	We argue by induction on $j$ that $G^{(j)}_i$ satisfies the hypotheses given in the statement of Lemma~\ref{lem:intermediate_description}.  Moreover, we argue by induction on $i$ that $F_i$ also satisfies the required hypotheses.
	
	First note that the hypotheses for $G^{(j)}_i$ are easily verified for any $i \in \{1,\ldots,l+1\}$ and $j=0$. Noting that
	\[
	F_1 := \frac{1}{z^{u^+_1}}\sum_{i > 1}(1,i),
	\]
	which is trivially of the required form, we have verified the hypotheses for $j < i = 1$. Considering $i=2$ we need to verify that $G^{(1)}_2$ and $F_2$ have the required form. Considering the expressions given in \eqref{eq:G12} and \eqref{eq:F2}, these also have the required form. We assume that $G^{(j)}_i$ and $F_k$ have the form required for some $j < i$ and all $k \in \{1,\ldots, i\}$. Let $P_G = (a_1\ldots, a_p)$ be a sequence corresponding to a term of $G^{(j)}_i$, and recall that $G^{(j+1)}_i = \mu^\star_{w_{j+1},F_{j+1}}G^{(j)}_i$. The sequence $P_G$ certainly satisfies $a_1 = i$. Moreover, if $a_p > j+1$ then, by Lemma~\ref{lem:evaluation}, we have that  $\mu^\star_{w_{j+1},F_{j+1}} P_G = P_G$. If however $a_p = j+1$, then $\mu^\star_{w_{j+1},F_{j+1}}P_G$ is a sum of terms $P$, where $P = (a_1,\ldots,a_p,b_1,\ldots, b_q)$ is given by concatenating $P_G$ with a sequence $P_F = (b_1,\ldots, b_q)$. Applying the inductive hypothesis, we have that $b_q > j+1$ and that the sequence corresponding to $P$ is hilly.
	
	We next check that $F_{i+1}$ has the required form. Note that since we have shown that $G^{(i-1)}_i$ has the required form, it remains to compute the terms which lie in $\tilde{G}^{(i-1)}_i$. It follows immediately from Lemma~\ref{lem:evaluation} that these are precisely those terms associated to sequences $(a_1,\ldots, a_p)$ where $a_p \neq a_1$. Since $a_1 = i$ and $a_p > j = i-1$, $F_i$ has the required form. We note that we can show that every sequence $P$ of the required class appears precisely once using a similar inductive argument. Indeed, the sequences corresponding to terms of $G^{(j)}_i$ either do not contain $(j-1)$, in which case the sequence appears in $G^{(j-1)}_i$, or they do so precisely once, in which case they are given by the concatenation of sequences appearing -- by the inductive hypothesis -- in $G^{(j-1)}_i$ and $F_j$ respectively.
\end{proof}

Note that the case $i=l+1$, $j=l$ of Lemma~\ref{lem:intermediate_description} gives a description of the polynomial $G^{(l)}_{l+1}$; we now describe the polynomials $G^{(l)}_i$ for $i \in \{1,\ldots l\}$.

\begin{pro}
	\label{pro:combinatorial_description}
	For each $i \in \{1,\ldots, l\}$, the polynomial $G^{(l)}_i$ has the form $z^{u^+_i} + \sum P$, where each polynomial $P$ is associated to a sequence $(a_1,\ldots a_p) \in \cM(i,i-1)$ such that $a_p = i$. Moreover, any $P$ of the class described appears in precisely one term of $G^{(l)}_i$.
\end{pro}
\begin{proof}
	We first note that, by Lemma~\ref{lem:evaluation}, we have that $G^{(l)}_i = G^{(i)}_i$. Thus we need to compute the polynomial $\mu^\star_{w_i,F_i}G^{(i-1)}_i$. Applying Lemma~\ref{lem:intermediate_description}, $G^{(i-1)}_i$ is a sum of terms of the form $P$, where $P$ is the polynomial associated to a sequence $(a_1,\ldots a_p) \in \cM(i,i-1)$. The mutation divides the terms such that $a_p > i$ by the factor polynomial $F_i$. This replaces the sum of all terms associated to sequences such that $a_p > i$ with the monomial $z^{u^+_i}$, and does not alter the remaining terms in $G^{(i-1)}_i$; that is, those associated to sequences such that $a_p = i$.
\end{proof}

\begin{proof}[Proof of Theorem~\ref{thm:alt_main_result}]
	We consider a diagram analogous to \eqref{eq:hypersurface_diagram}:
	\[
		\xymatrix{
			& \TT_N \ar^{ T_{u_c^+,u_c^-} \circ \cdots \circ T_{u_1^+,u_1^-}}[rrr] \ar@{-->}_{\mu_{w_1,F_1}\circ \cdots\circ\mu_{w_c,F_c}}[dl]  & & & \TT_N \ar@{-->}^{\mu_{-w_1,F'_1}\circ \cdots\circ\mu_{-w_c,F'_c}}[dr] & \\
			\TT_N  & & & & & \TT_N \ar@{-->}_{\varphi_C}[lllll]
		}
	\]
	We define $F'_j$ analogously to $F_i$, interchanging the roles of $S_i$ and $S'_i$. In particular, we also define $G'_i := \sum_{j \in S'_i}{z^{v_j}}$, ${G'}^{(0)}_i := G'_i$, and ${G'}^{(j)}_i := \mu^\star_{-w_j,F'_j}{G'}^{(j-1)}$. We note that it suffices to verify that $G^{(l)}_i$ and ${G'}^{(l)}_i$ are identified by $T := T_{u_l^+,u_l^-} \circ \cdots \circ T_{u_1^+,u_1^-}$ for any $i \in \{1,\ldots,l+1\}$. Indeed, recalling that $W_C = \sum_{i=1}^{l+1}{G_i}$, we have that
	\begin{align*}
	 W_C^{(l)} &= (\mu_{w_1,F_1}\circ \cdots\circ\mu_{w_c,F_c})^\star W_C \\ &= \sum_{i=1}^{l+1}{G^{(l)}_i}.
	\end{align*}
	
	We note that, by Lemma~\ref{lem:intermediate_description} and Proposition~\ref{pro:combinatorial_description}, the polynomials of $G^{(l)}_i$ are sums of polynomials corresponding to sequences $(a_1,\ldots, a_p)$, and that $a_1=a_p$ for all $i \in C$ (together with the monomial $z^{u^+_i}$ if $i \neq l+1$). We recall that, by definition,
	\[
	(a_1,\ldots, a_p) = z^{u^+_{a_1}}[a_1,a_2][a_2,a_3]\cdots [a_{p-1},a_p].
	\]
	Recalling that $T_{u_i^+,u_i^-}$ fixes each $z^{u^+_j}$ for $j \neq i$, and that $T_{u_i^+,u_i^-}(z^{u^\pm_i})= z^{u^\mp_i}$, it remains to show that, up to exchanging vectors $u^+_j$ and $u^-_j$, if the polynomial associated to a sequence $(a_1,\ldots, a_p)$ appears in a term of ${G'}^{(l)}_i$, then $P$ also appears as a term in $G^{(l)}_i$. Exchanging the roles of $S_i$ and $S'_i$, writing
	\[
		[i,j]' := \frac{1}{z^{u^-_i}}\sum_{k \in S'_i \cap S_j}{z^{v_k}},
	\]
	and applying Proposition~\ref{pro:combinatorial_description}, sequences $P'$ appearing in terms of ${G'}^{(l)}_i$ are of the form
	\[
	P' = z^{u^-_{a_1}}[a_1,a_2]'[a_2,a_3]'\cdots [a_{p-1},a_p]'.
	\]
	We thus need to show that terms of the form
	\[
	P = z^{u^+_{a_p}}[a_2,a_1][a_3,a_2]\cdots [a_p,a_{p-1}]
	\]
	appear as terms in $G^{(l)}_i$ (noting that the initial factor of $z^{u^+_{a_p}}$ is required to obtain the correct denominator, see Notation~\ref{not:condensed_polynomials}), and that all terms of $G^{(l)}_i$ are of this form. However this is easily verified: changing the order of multiplication, we have that
	\[
	P =  z^{u^+_{a_p}}[a_p,a_{p-1}][a_{p-1},a_{p-2}]\cdots [a_2,a_1] = (a_p,\ldots, a_1)
	\]
	and the result follows from the fact that $a_p=a_1$, and that the class of hilly sequences is palindromic; in the sense that if $(a_p,\ldots, a_1)$ is a hilly sequence, so is $(a_1,\ldots, a_p)$.
\end{proof}

%% file: amenable.tex
\section{Amenable collections}
\label{sec:amenable}

We recall the notion of an \emph{amenable collection subordinate to a nef partition}, and describe how to use algebraic mutations to form a Laurent polynomial mirror from a nef partition together with an amenable collection.

\begin{dfn}[C.f.\ {\cite[Definition~$2.2$]{Doran--Harder}}]
	Given a nef partition $\{S_1,\ldots,S_c\}$ of $(Y; L_1,\ldots, L_c)$, an amenable collection of vectors subordinate to $\{S_1,\ldots,S_c\}$	is a set of vectors $V = \{u_1,\ldots, u_c\} \subset M$ such that
	\begin{enumerate}
		\item For each $i \in \{1,\ldots, c\}$, we have $\langle u_i,S_i \rangle = -1$.
		\item For each $j$ such that $i < j < k$, we have $\langle u_i,S_j \rangle \geq 0$.
		\item For each $j$ such that $0 < j < i$, we have $\langle u_i,S_j \rangle = 0$.
	\end{enumerate}
\end{dfn}

Given such collection of vectors we can form a mirror Laurent polynomial as follows. First, we recall that $\breve{X} \subset \TT_N$ is defined to be the vanishing locus of polynomials $(G_i-1)$, where $G_i := \sum_{j \in S_i}{z^{v_j}}$ for $i \in \{1,\ldots, c\}$. Fixing distinguished elements $s_i \in S_i$ for each $i \in \{1,\ldots, c\}$, we consider a sequence of mutations defined as follows: for each $i$ we fix the weight vector $u_i$, and a factor polynomial $F_i$. The polynomials $F_i$ are inductively defined by setting $G^{(0)}_i := G_i$ for all $i \in \{1,\ldots, c\}$ and setting
\[
F_i := \frac{1}{z^{v_{s_i}}}G^{(i-1)}_i,
\]
where $G^{(j)}_i := \mu^\star_{w_j,F_j}G^{(j-1)}_i$. We let
\[
\breve{X}_1 := \{z \in \TT_N : z^{v_{s_1}} = \cdots = z^{v_{s_c}} = 1\},
\]
and note that the birational map $\phi$ formed by composing the mutations $\mu^\star_{w_i,F_i}$ for $i \in \{1,\ldots, c\}$ maps $\breve{X}_1$ to $\breve{X}$.

\begin{pro}
	\label{pro:residue}
	The restriction of the residue of the logarithmic form 
	\[
	\Omega := \mu^\star_{u_c,F_c}\circ \cdots \circ \mu^\star_{u_1,F_1} \Big( \frac{\omega_n}{(G_1-1)\cdots (G_c-1)} \Big)
	\]
	to $\breve{X}_1$ is the holomorphic volume form  $\omega_{n-c} = \bigwedge^{n-c}_{i=1}d\log(y_i)$, where the functions $y_i$ are co-ordinates on $\breve{X}_1 \cong (\CC^\star)^{n-c}$.
\end{pro}
\begin{proof}
Pulling back along these mutations, and recalling that
\[
\omega_n := \bigwedge^n_{i=1}d\log(z_i),
\]
we find that
\[
\Omega= \frac{\omega_n}{(z^{v_{s_1}}-1)\cdots (z^{v_{s_c}}-1)}.
\]
By \cite[Proposition~$2.3$]{Doran--Harder}, $u_1, \ldots, u_c$ can be extended to a basis $u_1, \ldots, u_n$ of $M$. Consider the vectors $\cB := \{v_{s_1}, \ldots, v_{s_c}\} \cup \{u^\star_{c+1},\ldots u^\star_n\}$; where the vectors $u^\star_j$ are elements of the basis of $N$ dual to $\{u_1, \ldots, u_n\}$. The matrix formed by pairing the elements of $\cB$ with elements of the basis $\{u_1, \ldots, u_n\}$ is lower triangular, with diagonal entries equal to $\pm 1$. That is, this matrix is unimodular, and hence $\cB$ is a basis of $N$. In particular we may relabel the variables $z_i$ such that $z_{s_i} = z^{v_{s_i}}$ for all $i \in \{1,\ldots, c\}$.

The residue of the form $\Omega$ along $\breve{X}_1$ is given by
\[
\frac{\omega_n}{dz_{s_1}\cdots dz_{s_c}}.
\]
Restricting to $\breve{X}_1$, this form coincides with the holomorphic form $\omega_{n-c}$.
\end{proof}

\begin{proof}[Proof of Corollary~\ref{cor:amenable_collection}]
	Fix a pair of nef partitions $\{S_1,\ldots, S_c\}$ and $\{S'_1,\ldots, S'_c\}$ and an amenable collection subordinate to each. We recall that $G_i = \sum_{j \in S_i}{z^{v_i}}$, and $G'_i = \sum_{j \in S'_i}{z^{v_i}}$. Let $\varphi$ be the birational map given by Theorem~\ref{thm:main_result} and observe that
	\[
	\varphi^\star\Big(\frac{\omega_n}{(1-G_1)\cdots(1-G_c)}\Big) = \frac{\omega_n}{(1-G'_1)\cdots(1-G'_c)}.
	\]
	Setting
	\[
	\breve{X}_2 := \{z \in \TT_N : z^{v_{s'_1}} = \cdots = z^{v_{s'_c}} = 1\},
	\]
	where $\phi'$ denotes the birational map -- analogous to $\phi$ -- from $\breve{X}_2$ to $\breve{X}$. Applying Proposition~\ref{pro:residue} to both $\phi$ and $\phi'$, we have that
	\[
	({\phi'}^{-1} \circ \varphi \circ \phi)^\star\omega_{n-c} = \omega_{n-c}.
	\]
	Moreover, since $\varphi^\star W = W$, this composition identifies the Laurent polynomials defined on $\breve{X}_1$ and $\breve{X}_2$ respectively.
\end{proof}